\documentclass{amsproc}
\usepackage{amsfonts}

\theoremstyle{plain}

\newtheorem{corollary}{Corollary}

\newtheorem{definition}{Definition}

\newtheorem{proposition}{Proposition}

\newtheorem{theorem}{Theorem}
\numberwithin{equation}{section}

\begin{document}
\title[Pre-Markov]{Pre-Markov Operators}
\author{H\={u}lya Duru}
\address{Istanbul University, Faculty of Science, Mathematics Department,
Vezneciler-Istanbul, 34134, Turkey }
\email{hduru@istanbul.edu.tr}
\author{Serkan Ilter}
\curraddr{Istanbul University, Faculty of Science, Mathematics Department,
Vezneciler-Istanbul, 34134, Turkey }
\email{ilters@istanbul.edu.tr}

\begin{abstract}
A positive linear operator $T$ between two unital $f$-algebras, with point
separating order duals, $A$ and $B$ is called a Markov operator for which $%
T\left( e_{1}\right) =e_{2}$ where $e_{1},e_{2}$ are the identities of $A$
and $B$ respectively. Let $A$ and $B$ be semiprime $f$-algebras with point
separating order duals such that their second order duals $A^{\sim \sim }$
and $B^{\sim \sim }$ are unital $f$-algebras. In this case, we will call a
positive linear operator $T:A\rightarrow B$ \ to be a pre-Markov operator,
if the second adjoint operator of $T$ is a Markov operator. A positive
linear operator $T$ between two semiprime $f$-algebras, with point
separating order duals, $A$ and $B$ is said to be contractive if $Ta\in
B\cap \left[ 0,I_{B}\right] $ whenever $a\in A\cap \left[ 0,I_{A}\right] $,
where $I_{A}$ and $I_{B}$ are the identity operators on $A$ and $B$
respectively. In this paper we characterize pre-Markov algebra
homomorphisms. In this regard, we show that a pre-Markov operator is an
algebra homomorphism if and only if its second adjoint operator is an
extreme point in the collection of all Markov operators from $A^{\sim \sim }$
to $B^{\sim \sim }$. Moreover we characterize extreme points of contractive
operators from $A$ to $B$. In addition, we give a condition that makes an
order bounded algebra homomorphism is a lattice homomorphism.
\end{abstract}

\maketitle

\section{\textbf{Introduction}}
A positive linear operator $T$ between two unital $f$-algebras, with point
separating order duals, $A$ and $B$ is called a Markov operator for which $%
T\left( e_{1}\right) =e_{2}$ where $e_{1}$, $e_{2}$ are the identities of $A$
and $B$ respectively. Let $A$ and $B$ be semiprime $f$-algebras with point
separating order duals such that their second order duals $A^{\sim \sim }$
and $B^{\sim \sim }$ are unital $f$-algebras. In this case, we will call a
positive linear operator $T:A\rightarrow B$ \ to be a pre-Markov operator,
if the second adjoint operator of $T$ is a Markov operator. A positive
linear operator $T$ between two semiprime $f$-algebras, with point
separating order duals, $A$ and $B$ is said to be contractive if $Ta\in
B\cap \left[ 0,I_{B}\right] $ whenever $a\in A\cap \left[ 0,I_{A}\right] $,
where $I_{A}$ and $I_{B}$ are the identity operators on $A$ and $B$
respectively.

The collection of all pre-Markov operators is a convex set. In this paper,
first of all, we characterize pre-Markov algebra homomorphisms. In this
regard, we show that a pre-Markov operator is an algebra homomorphism if and
only if its second adjoint operator is an extreme point in the collection of
all Markov operators from $A^{\sim \sim }$ to $B^{\sim \sim }$ (Theorem \ref{t1}%
). In addition, we characterize the extreme points of all contractive
operators $T:A\rightarrow B$ whenever $A$ and $B$ are Archimedean semiprime
f-algebras provided $B$ is relatively uniformly complete (Proposition \ref{p3}). For the second aim, let $A$ and $B$ be Archimedean semiprime f-algebras and $%
T:A\rightarrow B$ a linear operator. Huijsman and de Pagter proved in \cite{Ref8}\textrm{\ }the following:

\noindent (1) If $T$ is a positive algebra homomorphism then it s a lattice
homomorphism.

\noindent (2) In addition, if the domain $A$ is relatively uniformly
complete and $T$ is an algebra homomorphism then it s a lattice homomorphism
and the assumption that the domain $A$ of $T$ is relatively uniformly
complete is not reduntant (Theorem $5.1$ and Example $5.2$.).

\noindent (3) In addition, if the domain $A$ has a unit element and $T$ is
an order bounded algebra homomorphism then it s a lattice
homomorphism (Theorem $5.3$).

We prove that any order bounded algebra homomorphism $T:A\rightarrow B$ is a
lattice homomorphism, if the region $B$ is relatively uniformly complete
(Corollary \ref{c3}). In this regard, first we give an alternate proof of Lemma $6$
in \cite{Ref10}\textrm{\ }for order bounded operators with the
relatively uniformly complete region instead of positive operators
with Dedekind complete region $\left( \text{Proposition }\ref{p4}\text{ and }%
\ref{p5}\right) $. In the last part, we give a necessary and sufficient condition for a
positive operator to be a lattice homomorphism (Proposition \ref{p7}).

\section{\textbf{Preliminaries}}

For unexplained terminology and the basic results on vector lattices and
semiprime $f$-algebras we refer to \cite{Ref1,Ref11,Ref13,Ref15}\textrm{. }The
real algebra $A$ is called a Riesz algebra or lattice-ordered algebra if $A$
is a Riesz space such that $ab\in A$ whenever $a,b$ are positive elements in
$A.$ The Riesz algebra is called an $f$-algebra if $A$ satisfies the
condition that
\begin{equation*}
a\wedge b=0\text{ implies }ac\wedge b=ca\wedge b=0\text{ for all }0\leq c\in
A\text{.}
\end{equation*}%
In an Archimedean $f$-algebra $A$, all nilpotent elements have index $2$.
Indeed, assume that $a^{3}=0$ for some $0\leq a\in A$. Since the equality $%
\left( a^{2}-na\right) \wedge \left( a-na^{2}\right) =0$ implies $\left(
a^{2}-na\right) \wedge a^{2}=\left( a^{2}-na\right) =0$ we get $a^{2}=0$ as $%
A$ is Archimedean. The same argument is true for all $n\geq 3$. Throughout
this paper \thinspace $A$ will show an Archimedean semiprime $f$-algebra
with point separating order dual $A^{\sim }$ \cite{Ref15}. By
definition, if zero is the unique nilpotent element of $A$, that is, $%
a^{2}=0 $ implies $a=0$, $A$ is called semiprime f-algebra. It is well known
that every $f$-algebra with unit element is semiprime.

Let $A$ be a lattice ordered algebra. If $A$ is a lattice ordered space,
then the first order dual space $A^{\sim }$ of $A$ is defined as the
collection of all order bounded linear functionals on $A$ and $A^{\sim }$ is
a Dedekind complete Riesz space. The second order dual space of $A$ is
denoted by $A^{\sim \sim }$. Let $a\in A$, $f\in A^{\sim }$ and $F,G\in
A^{\sim \sim }$. Define $f\cdot a\in A^{\sim }$, by%
\begin{equation*}
\left( f\cdot a\right) \left( b\right) =f\left( ab\right)
\end{equation*}%
and $F\cdot f\in A^{\sim }$, by%
\begin{equation*}
\left( F\cdot f\right) \left( a\right) =F\left( f\cdot a\right)
\end{equation*}%
and $F\cdot G\in A^{\sim \sim }$, by%
\begin{equation*}
\left( F\cdot G\right) \left( f\right) =F\left( G\cdot f\right)
\end{equation*}%
\bigskip The last equality is called the Arens multiplication in $A^{\sim
\sim }$ \cite{Ref2}.

The second order dual space $A^{\sim \sim }$ of a semiprime $f$-algebra $A$
is again an $f$-algebra with respect to the Arens multiplication \cite{Ref4}\textrm{. }In the literature, there are several studies, for
example \cite{Ref5,Ref6,Ref7,Ref9}, that respond the question "Under what
conditions does the $f$-algebra $A^{\sim \sim }$ have a unit element?".

Let $A$ and $B$ be semiprime $f$-algebras with point separating order duals
such that their second order duals $A^{\sim \sim }$ and $B^{\sim \sim }$
have unit elements $E_{1}$ and $E_{2}$ respectively. Let $T:A\rightarrow B$
be an order bounded operator. We denote the second adjoint operator of $T$
by $T^{\ast \ast }$. Since $A$ and $B$ have point separating order duals,
the linear operator $J_{1}:A\rightarrow A^{\sim \sim }$, which assigns to $%
a\in A$ the linear functional $\widehat{a}$ defined on $A^{\sim }$ by $%
\widehat{a}\left( f\right) =f\left( a\right) $ for all $a\in A$, is an
injective algebra homomorphism. Therefore we will identify $A$ with $%
J_{1}\left( A\right) $, and $B$ with $J_{2}\left( B\right) $ in the similar
sense.

\begin{definition}
Let $A$ and $B$ be semiprime $f$-algebras with point separating order duals
such that their second order duals $A^{\sim \sim }$ and $B^{\sim \sim }$ are
unital $f$-algebras. In this case, we call a positive linear operator $%
T:A\rightarrow B$ \ to be a pre-Markov operator, if the second adjoint
operator of $T$ is a Markov operator. That is, the second adjoint operator $%
T^{\ast \ast }:A^{\sim \sim }\rightarrow B^{\sim \sim }$ of $T$ is a
positive linear and $T^{\ast \ast }\left( E_{1}\right) =E_{2}$, where $E_{1}
$ and $E_{2}$ are the unitals of $A$ and $B$ respectively.
\end{definition}

Recall that a positive operator $T:A\rightarrow B$ satisfying $0\leq T\left(
a\right) \leq E_{2}$ whenever $0\leq a\leq E_{1}$ is called a contractive
operator.

In this point we remark that , if $A$ and $B$ are semiprime $f$-algebras
with point separating order duals and $T:A\rightarrow B$ \ is a positive
linear operator, then $T^{\ast \ast }$ is positive. Indeed, let $0\leq F\in
A^{\sim \sim }$ and $0\leq g\in B^{\sim }$. Then $0\leq g\circ T\in A^{\sim
} $ and therefore $F\left( g\circ T\right) =T^{\ast \ast }\left( F\right)
\geq 0$.

\begin{proposition}
\label{p1}Let $A$ and $B$ be semiprime $f$-algebras with point separating
order duals such that their second order duals $A^{\sim \sim }$ and $B^{\sim
\sim }$ have unit elements $E_{1}$ and $E_{2}$ respectively. $T:A\rightarrow
B$ is contractive if and only if $T^{\ast \ast }$ is contractive.
\end{proposition}

\begin{proof}
Suppose that $T$ is contractive. Then $T^{\ast \ast }$ is positive. Let $%
F\in \left[ 0,E_{1}\right] \cap A^{\sim \sim }$. In order to prove that $%
T^{\ast \ast }$ is contractive we shall show that $T^{\ast \ast }\left(
E_{1}\right) \leq E_{2}$. Due to\textit{\ }\cite{Ref9},
\begin{eqnarray*}
E_{1}\left( f\right) &=&\sup f\left( A\cap \left[ 0,E_{1}\right] \right) \\
E_{2}\left( g\right) &=&\sup g\left( B\cap \left[ 0,E_{2}\right] \right)
\end{eqnarray*}%
for all $f$ $\in A^{\sim }$ and $g\in B^{\sim }$. Let $a\in A\cap \left[
0,E_{1}\right] $ and $0\leq g\in B^{\sim }$. Since $T$ is contractive, $%
T\left( a\right) \in B\cap \left[ 0,E_{2}\right] $ so $g\left( T\left(
a\right) \right) \leq E_{2}\left( g\right) $ which implies that $T^{\ast
\ast }E_{1}\left( g\right) =E_{1}\left( g\circ T\right) \leq E_{2}\left(
g\right) $. Thus $T^{\ast \ast }\left( E_{1}\right) \leq E_{2}$. Conversely,
assume that $T^{\ast \ast }$ is contractive. Let $a\in A\cap \left[ 0,E_{1}%
\right] $ and $0\leq g\in B^{\sim }$. Then $\widehat{Ta}\left( g\right)
=g\left( Ta\right) \leq T^{\ast \ast }E_{1}\left( g\right) \leq E_{2}\left(
g\right) $ Thus $0\leq Ta=\widehat{Ta}\leq E_{2}$.
\end{proof}

\begin{corollary}
\label{c1}Let $A$ and $B$ be semiprime $f$-algebras with point separating
order duals such that their second order duals $A^{\sim \sim }$ and $B^{\sim
\sim }$ have unit elements $E_{1}$ and $E_{2}$ respectively. If $\
T:A\rightarrow B$ is a pre-Markov operator then $T$ is contractive.
\end{corollary}

\begin{proof}
Since $T^{\ast \ast }\left( E_{1}\right) =E_{2}$ and $T^{\ast \ast }$ is
positive, $T^{\ast \ast }$ is contractive. By Proposition \ref{p1} we have the
conclusion.
\end{proof}

For $a\in A$, the mapping $\pi _{a}:A\rightarrow OrthA,$ defined by $\pi
_{a}\left( b\right) =a.b$ is an orthomorphism on $A$. Since $A$ is a
Archimedean semiprime $f$-algebra, the mapping $\pi :A\rightarrow OrthA$,$\,$
defined by $\pi \left( a\right) =\pi _{a}$ is an injective $f$-algebra
homomorphism. Hence we shall identify $A$ with $\pi \left( A\right) $.

\section{Main Results}

\begin{theorem}
\label{t1}Let $A$ and $B$ be semiprime $f$-algebras with point separating
order duals such that their second order duals $A^{\sim \sim }$ and $B^{\sim
\sim }$ have unit elements $E_{1}$ and $E_{2}$ respectively. A pre-Markov
operator $T:A\rightarrow B$ is an algebra homomorphism if and only if its
second adjoint operator $T^{\ast \ast }$ is an algebra homomorphism.
\end{theorem}

\begin{proof}
Suppose that the pre-Markov operator $T$ is an algebra homomorphism. Since $%
T^{\ast \ast }$ is a Markov operator, due to \cite{Ref8}, it is
enough to show that it is a lattice homomorphism. Let $F,G\in A^{\sim \sim }$
such that $F\wedge G=0$. Since $A^{\sim \sim }$ and $B^{\sim \sim }$ are
semiprime $f$- algebras, $F\cdot G=0$. We shall show that $T^{\ast \ast
}\left( F\right) \cdot T^{\ast \ast }\left( G\right) =0$. Let $a,b\in A$ and
$f\in B^{\sim }$. Then it follows from the following equations%
\begin{eqnarray*}
\left( \left( f\cdot Ta\right) \circ T\right) \left( b\right) &=&\left(
f\cdot Ta\right) \left( Tb\right) =f\left( TaTb)=f(T(ab)\right) \\
&=&\left( f\circ T\right) \left( ab\right) =\left( (f\circ T)\cdot a\right)
\left( b\right)
\end{eqnarray*}%
that
\begin{equation}\label{eq1}
\left( f\cdot Ta\right) \circ T=\left( f\circ T\right) \cdot
a.
\end{equation}%
On the other hand, the following equations%
\begin{equation*}
\left( \left( G\circ T^{\ast })\cdot f\right) \circ T\right) \left( a\right)
=(\left( G\circ T^{\ast })\cdot f\right) \left( Ta\right) =(G\circ T^{\ast
})\left( f\cdot Ta\right) =G\left( (f\cdot Ta)\circ T\right)
\end{equation*}%
hold. Thus $\left( \left( G\circ T^{\ast })\cdot f\right) \circ T\right)
\left( a\right) =G\left( (f\cdot Ta)\circ T\right) $. From here, by setting
the equation (\ref{eq1}), we conclude that%
\begin{equation*}
\left( \left( G\circ T^{\ast })\cdot f\right) \circ T\right) \left( a\right)
=G\left( \left( f\circ T\right) \cdot a\right) =\left( G\cdot (f\circ
T)\right) \left( a\right)
\end{equation*}%
which implies%
\begin{equation}\label{eq2}
\left( \left( G\circ T^{\ast })\cdot f\right) \circ T\right) =\left( G\cdot
\left( f\circ T\right) \right).
\end{equation}%
Taking into account the equation (\ref{eq2}), we get
\begin{eqnarray*}
\left( T^{\ast \ast }\left( F\right) \cdot T^{\ast \ast }\left( G\right)
\right) \left( f\right) &=&T^{\ast \ast }\left( F\right) (\left( T^{\ast
\ast }\left( G\right) \cdot f\right) )=\left( F\circ T^{\ast }\right)
(\left( G\circ T^{\ast })\cdot f\right) \\
&=&F\left( \left( G\circ T^{\ast })\cdot f\right) \circ T\right) =F\left(
G\cdot \left( f\circ T\right) \right)
\end{eqnarray*}%
thus we have
\begin{equation*}
\left( T^{\ast \ast }\left( F\right) \cdot T^{\ast \ast }\left( G\right)
\right) \left( f\right) =\left( F\cdot G\right) \left( f\circ T\right) =0
\end{equation*}%
as desired. Conversely suppose that $T^{\ast \ast \text{ }}$is an algebra
homomorphism. Let $a,b\in A$. It follows from%
\begin{equation*}
T(ab)=\widehat{T(ab)}=T^{\ast \ast }\left( \widehat{ab}\right) =T^{\ast \ast
}\left( \widehat{a}\cdot \widehat{b}\right) =T^{\ast \ast }\left( \widehat{a}%
\right) \cdot T^{\ast \ast }\left( \widehat{b}\right) =\widehat{Ta}\cdot
\widehat{Tb}=Ta.Tb
\end{equation*}%
that $T$ is an algebra homomorphism.
\end{proof}

In the proof of Theorem \ref{t1} we proved the following corollary as well;

\begin{corollary}
\label{c2}Let $A$, $B$ and their second order duals $A^{\sim \sim }$ and $%
B^{\sim \sim }$ be semiprime $f$-algebras and $T:A\rightarrow B$ a positive
algebra homomorphism. Then $T^{\ast \ast }$ is a lattice homomorphism.
\end{corollary}

\begin{theorem}
\label{t2}Let $A$ and $B$ be semiprime $f$-algebras with point separating
order duals and $T:A\rightarrow B$ a positive linear operator. If the second
order duals $A^{\sim \sim }$ and $B^{\sim \sim }$ have unit elements and $T$
is an algebra homomorphism, then $T$ is an extreme point of the contractive
operators from $A$ to $B$.
\end{theorem}

\begin{proof}
$~$Suppose that $T$ is a positive algebra homomorphism. Then due to \cite[Theorem~4.3] {Ref14}, $T$ is a contractive operator. Let $2T=T_{1}+T_{2}$ for
some contractive operators $T_{1},T_{2}$ from $A$ to $B$. In this case, $%
2T^{\ast \ast }=T_{1}^{\ast \ast }+T_{2}^{\ast \ast }$. By Proposition \ref{p1}, $%
T^{\ast \ast }$, $T_{1}^{\ast \ast }$ and $T_{2}^{\ast \ast }$ are
contractive and by Corollary \ref{c2}, $T^{\ast \ast }$ is a lattice
homomorphism. Taking into account \cite[Theorem~3.3] {Ref3}, we
derive that $T^{\ast \ast }$ is an extreme point in the collection of all
contractive operators from $A^{\sim \sim }$ to $B^{\sim \sim }$. Thus $%
T^{\ast \ast }=T_{1}^{\ast \ast }=T_{2}^{\ast \ast }$ and therefore $%
T=T_{1}=T_{2}$.
\end{proof}

At this point, we recall the definition of uniform completion of an
Archimedean Riesz space. If $A$ is an Archimedean Riesz space and $\widehat{A%
}$ is the Dedekind completion of $A$, then $\overline{A},$ the closure of $A$
in $\widehat{A}$ with respect to the relatively uniform topology \cite{Ref11}\textrm{, }is so called that relatively uniformly completion of $A$
\cite{Ref12}\textrm{. }If $A$ is an semiprime $f$-algebra then the
multiplication in $A$ can be extended in a unique way into a lattice ordered
algebra multiplication on $\overline{A}$ such that $A$ becomes a sub-algebra
of $\overline{A}$ and $\overline{A}$ is an relatively uniformly complete
semiprime $f$-algebra. In \cite[Theorem~3.4] {Ref14} it is shown that a positive
operator $T$ from a Riesz space $A$ to a uniformly complete space $B$, has a
unique positive linear extension $\ \overline{T}:\overline{A}\rightarrow B$
to the relatively uniformly completion $\overline{A}$ of $A$, defined by,

\begin{equation*}
\overline{T}\left( x\right) =\sup \left\{ T(a):0\leq a\leq x\right\}
\end{equation*}%
for $0\leq $ $x\in \overline{A}$. We also recall that $\overline{A}$
satisfies the Stone condition (that is, $x\wedge nI^{\ast }\in \overline{A}$%
, for all $x\in \overline{A}$, where $I$ denotes the identity on $A$ of $%
OrthA$) due to Theorem $2.5$ in \cite{Ref7}. For the completeness we give
the easy proof of the following proposition.

\begin{proposition}\label{p2}
Let $A$ and $B$ be Archimedean semiprime $f$-algebras such that $%
B $ is relatively uniformly complete. In this case, $T:A\rightarrow B$ is
contractive if and only if $\ \overline{T}$ is contractive.
\end{proposition}

\begin{proof}
Suppose that $T$ is contractive. Let $x\in \overline{A}\cap \left[ 0,%
\overline{I}\right] $, here $\overline{I}$ is the unique extension to $%
\overline{A}$ of the identity mapping $I:A\rightarrow A$. Since $T$ is
contractive, $a\in A\cap \left[ 0,x\right] $ implies that $I$ is an upper
bound for the set $\left \{ T\left( a\right) :a\leq x,a\in A\right \} $, so $%
\overline{T}\left( x\right) \leq I$. Therefore $\overline{T}$ \ is
contractive. The converse implication is trivial, since $\overline{T}$ is
the extension of $T$, we get $0\leq \overline{T}\left( a\right) =T\left(
a\right) \leq I$ whenever $a\in A\cap \left[ 0,I\right] $.
\end{proof}

\begin{proposition}
\label{p3}Let $A$ and $B$ be Archimedean semiprime $f$- algebras such that $%
B $ is relatively uniformly complete and let $T:A\rightarrow B$ be a
contractive operator. Then $T$ is an extreme point in the collection of all
contractive operators from $A$ to $B$ if and only if $\ \overline{T}$ is an
extreme point of all contractive operators from $\overline{A}$ to $B$.
\end{proposition}

\begin{proof}
Suppose that $\overline{T}$ is an extreme point in the set of all
contractive operators from $\overline{A}$ to $B$. We shall show that for
arbitrary $\varepsilon >0$ and contractive operator $S$ from $A$ to $B$
satisfying $\varepsilon T-S\geq 0$ implies that $T=S$. Let $0\leq x\in
\overline{A}$. Then there exists a positive sequence $\left( a_{n}\right)
_{n}$ in $A_{\text{ }}$ converging relatively uniformly to $x$. Since $%
\overline{T}$ and $\overline{S}$ are relatively uniformly continuos, the
sequence $\varepsilon \overline{T}\left( a_{n}\right) $ $-\overline{S}\left(
a_{n}\right) =\varepsilon T\left( a_{n}\right) -S\left( a_{n}\right) $
converges to $\varepsilon \overline{T}\left( x\right) $ $-\overline{S}\left(
x\right) $. Therefore, since $\left( a_{n}\right) _{n}$ is positive sequence
and $\varepsilon T-S\geq 0$, we get $\varepsilon \overline{T}$ $-\overline{S}%
\geq 0$. Since $\overline{T}$ is an extreme point, we have $\overline{T}=%
\overline{S}$, so that $T=S$. Conversely assume that $T$ is an extreme point
in the set of all contractive operators from $A$ to $B$. Let $\varepsilon >0$
be any number and let $S$ be any contractive operator from $\overline{A}$ to $B$
satisfying $\varepsilon \overline{T}$ $-S\geq 0$. Let $U$ be the restriction
of $S$ to $A$. Since $S$ is contractive, by Proposition \ref{p2}, $S\mid _{A}=U$
is contractive and by the uniquness of the extension, we infer that $S=%
\overline{U}$. Hence $\left( \varepsilon \overline{T}-S\right) \mid
_{A}=\varepsilon T-U\geq 0$. Thus $\overline{T}=S$, which shows that $%
\overline{T}$ is an extreme point.
\end{proof}

After proving the following Propositions \ref{p4} and \ref{p5} for order bounded
operators with the relatively uniformly complete region, we remarked that
both were proved in \cite{Ref10} for the positive operators with
Dedekind complete region. They might be regarded as the alternate proofs.

\begin{proposition}
\label{p4}Let $T:A\rightarrow B$ be an order bounded operator where $A$ and $%
B$ are Archimedean $f$-algebras and $B$ is, in addition, relatively
uniformly complete. Then $T$ is an algebra homomorphism\textbf{\ }iff $\
\overline{T}$ is an algebra homomorphism.
\end{proposition}

\begin{proof}
Suppose that $T:A\rightarrow B$ is an algebra homomorphism and $x,y$ be
positive elements in $\overline{A}$. By \cite{Ref14}, since%
\begin{equation*}
xy=\sup \left\{ R_{y}\left( a\right) :0\leq a\leq x,a\in A\right\}
\end{equation*}%
and
\begin{equation*}
R_{y}\left( a\right) =\sup \left\{ ab:0\leq b\leq y,b\in A\right\} \text{.}
\end{equation*}%
Now as $\overline{T}$ is relatively uniformly continuous, we get,
\begin{align*}
\overline{T}\left( R_{y}\left( a\right) \right) & =\sup \left\{ \overline{T}%
\left( ab\right) =T\left( ab\right) =T\left( a\right) T\left( b\right)
:0\leq b\leq y,b\in A\right\} \\
& =T\left( a\right) \sup \left\{ T\left( b\right) :0\leq b\leq y,b\in
A\right\} \\
& =T\left( a\right) \overline{T}\left( y\right)
\end{align*}%
and then
\begin{align*}
\overline{T}\left( xy\right) & =\sup \left\{ \overline{T}\left( R_{y}\left(
a\right) \right) :0\leq a\leq x,a\in A\right\} \\
& =\sup \left\{ T\left( a\right) \overline{T}\left( y\right) :0\leq a\leq
x,a\in A\right\} \\
& =\overline{T}\left( y\right) \sup \left\{ T\left( a\right) :0\leq a\leq
x,a\in A\right\} \\
& =\overline{T}\left( x\right) \overline{T}\left( y\right)
\end{align*}%
Hence $\overline{T}$ is an algebra homomorphism. The converse is trivial.
\end{proof}

In \cite{Ref8}, both were proved that an algebra homomorphism $%
T:A\rightarrow B$ need not be a lattice homomorphism if the domain $A$ is
not relatively uniformly complete $\left( \text{Example }5.2\right) $ and an
order bounded algebra homomorphism $T:A\rightarrow B$ is a lattice
homomorphism whenever the domain $A$ has a unit element. We remarked that
Proposition \ref{p4} yields that the second result also holds for an order
bounded algebra homomorphism without unitary domain but the region is
relatively uniformly complete.

\begin{corollary}
\label{c3}Let $A$ be an Archimedean semiprime $f$-algebra and $B$ a
relatively uniformly complete Archimedean $f$-algebra. Then any order
bounded algebra homomorphism $T:A\rightarrow B$ is a lattice homomorphism.
\end{corollary}

\begin{proof}
By Proposition \ref{p4}, $\overline{T}$ is an algebra homomorphism and since $%
\overline{A}$ is relatively uniformly complete, $\overline{T}$ \ is a
lattice homomorphism \cite{Ref8}\textrm{.} Thus $T$ is a lattice
homomorphism.
\end{proof}

\begin{proposition}
\label{p5}Let $A$ be an Archimedean $f$-algebra and let $B$ be a relatively
uniformly complete semiprime $f$- algebra. Then the operator $T:A\rightarrow
B$ is a lattice homomorphism\textbf{\ }iff $\ \overline{T}$ is a lattice
homomorphism.
\end{proposition}

\begin{proof}
Suppose that $T$ is a lattice homomorphism. Let $x\in \overline{A}$. Let $%
a\in \left[ 0,x^{+}\right] \cap A$ and $b\in \left[ 0,x^{-}\right] \cap A$.
Since $T$ is a lattice homomorphism, we have
\begin{equation*}
T\left( a\wedge b\right) =T\left( a\right) \wedge T\left( b\right) =0\text{.}
\end{equation*}%
On the other hand, it follows from the equality
\begin{equation*}
T\left( a\right) \wedge \overline{T}\left( x^{-}\right) =\sup \left\{
T\left( a\right) \wedge T\left( b\right) :0\leq b\leq x^{-},b\in A\right\}
\end{equation*}%
that%
\begin{equation*}
\overline{T}\left( x^{+}\right) \overline{T}\left( x^{-}\right) =\sup
\left\{ T\left( a\right) \wedge \overline{T}\left( x^{-}\right) :0\leq a\leq
x^{+},a\in A\right\} =0
\end{equation*}%
which its turn is equivalent to $\overline{T}$ is a lattice homomorphism, as
$B$ is semiprime. Converse is trivial.
\end{proof}

In this point, we remark that Lemma 3.1 and Theorem 3.3 in \cite{Ref3} are also true for Archimedean semiprime $f$-algebras without the Stone
condition on the domain $A$ whenever $B$ is relatively uniformly complete.

\begin{proposition}
\label{p6}Let $A$ and $B$ be Archimedean semiprime $f$-algebras, $B$
relatively uniformly complete and $T:A\rightarrow B$ a contractive operator.
Assume that $\overline{A}$ has unit element. For $y\in \overline{A}$, define
$H_{x}\left( y\right) =\overline{T}\left( xy\right) -\overline{T}\left(
x\right) \overline{T}\left( y\right) $. Then $\overline{T}\overset{\_}{+}%
H_{x}$ are contractive mappings for all $x\in \overline{A}\cap \left[ 0,I%
\right] $.
\end{proposition}

\begin{proof}
By Proposition \ref{p2}, $\overline{T}$ is contractive. Since $\overline{A}$
satisfies the Stone condition, due to \cite[Lemma~3.1] {Ref3},
we have the conclusion.
\end{proof}

\begin{corollary}
\label{c4}Let $A$ and $B$ be Archimedean semiprime $f$-algebras such that $B$
is relatively uniformly complete and let $T:A\rightarrow B$ be a contractive
operator. If $\overline{A}$ has unit element, then \thinspace $T$ $\overset{%
\_}{+}T_{a}$ are contractive for all $a\in A\cap \left[ 0,I\right] $, here $%
T_{a}\left( b\right) =T\left( ab\right) -T\left( a\right) T\left( b\right) $.
\end{corollary}

\begin{proof}
By Proposition \ref{p6}, $\overline{T}\overset{\_}{+}H_{x}$ are contractive
mappings for all $x\in \overline{A}\cap \left[ 0,I\right] $. Let $a\in A\cap %
\left[ 0,I\right] $ and $0\leq b\in A$. Then $0\leq \left( \overline{T}%
\overset{\_}{+}H_{a}\right) $ $\left( b\right) =T\left( b\right) \overset{\_}%
{+}T_{a}\left( b\right) $ holds. Thus $T\overset{\_}{+}T_{a}$ is positive.
Let $b\in A\cap \left[ 0,I\right] $.$\,$It follows from%
\begin{equation*}
0\leq \left( \overline{T}\overset{\_}{+}H_{a}\right) \left( b\right)
=T\left( b\right) \overset{\_}{+}T_{a}\left( b\right) \leq I
\end{equation*}%
that $T$ $\overset{\_}{+}T_{a}$ are contractive.
\end{proof}

\begin{proposition}
\label{p7}Let $A$ and $B$ be Archimedean semiprime $f$-algebras such that $B$
is relatively uniformly complete and let $T:A\rightarrow B$ be a positive
linear operator. $T$ is contractive and it is an extreme point in the
collection of all contractive operators from $A$ to $B$ if and only if $T$
is an algebra homomorphism.
\end{proposition}

\begin{proof}
Let $T$ be an extreme point in the collection of all contractive operators
from $A$ to $B$. Then by Proposition \ref{p3}, $\overline{T}$ is an extreme point
of all contractive operators from $\overline{A}$ to $B$. It follows from \cite[Theorem~3.3] {Ref3} that $\overline{T}$ is an algebra
homomorphism. By Proposition \ref{p4}, $T$ is an algebra homomorphism.
Conversely, if $T$ is an algebra homomorphism, then due to \cite[Theorem~4.3] {Ref14}, $T$ is a contractive operator. By Proposition \ref{p2}, $%
\overline{T}$ is contractive and by Proposition \ref{p4}, $\overline{T}$ is an
algebra homomorphism. Thus $\overline{T}$ is an extreme point in the set of
all contractions from $\overline{A}$ to $B$ due to \cite[Theorem~3.3] {Ref3}. By using Proposition \ref{p3}, we have the conclusion.
\end{proof}

\end{document}